\documentclass[12pt,a4paper]{amsart}
\usepackage[utf8]{inputenc}
\usepackage{amsmath,amscd, amsfonts, amssymb, amsthm}
\usepackage{parskip}
\usepackage{fullpage}
\usepackage[colorlinks, breaklinks, allcolors=blue]{hyperref} 
\usepackage{breakurl}
\usepackage[all]{xypic}
\usepackage{ctable}
\usepackage{longtable}
\usepackage{array}
\usepackage{booktabs}
\usepackage{float}
\floatstyle{plaintop}
\restylefloat{table}

\newcommand\CB{{\mathcal B}}

\newcommand\CG{{\mathcal G}}
\newcommand\CGtilde{{\widetilde\CG}}
\newcommand\CH{{\mathcal H}} 
 
\newcommand\CL{{\mathcal L}}

\newcommand\CO{{\mathcal O}}

\newcommand\CT{{\mathcal T}} 
\newcommand\CU{{\mathcal U}} 
\newcommand\CV{{\mathcal V}}
\newcommand\CX{{\mathcal X}}
\newcommand\CY{{\mathcal Y}}

\newcommand\fbar{{\overline f}}

\newcommand\Htilde{{\widetilde H}}
\newcommand\Ltilde{{\widetilde L}}

\newcommand\Qtilde{{\widetilde Q}}

\newcommand\BBA{{\mathbb A}}
\newcommand\BBC{{\mathbb C}}
\newcommand\BBF{{\mathbb F}}
\newcommand\BBG{{\mathbb G}}
\newcommand\BBN{{\mathbb N}}
\newcommand\BBP{{\mathbb P}}
\newcommand\BBQ{{\mathbb Q}}
\newcommand\BBR{{\mathbb R}}
\newcommand\BBZ{{\mathbb Z}}

\newcommand\id{{\mathrm {id}}}
\newcommand\gen{{\mathrm {gen}}}
\newcommand\soc{{\mathrm {soc}}}

\renewcommand\mod{{/\!\!/}}

\newcommand\A{{\operatorname A}}
\newcommand\B{{\operatorname B}}
\newcommand\C{{\operatorname C}}
\newcommand\D{{\operatorname D}}
\newcommand\E{{\operatorname E}}
\newcommand\F{{\operatorname F}}
\newcommand\G{{\operatorname G}}
\newcommand\Gm{{{\BBG}_m}}

\newcommand\Char{\operatorname{char}}

\newcommand\GL{\operatorname{GL}}

\renewcommand\ker{{\operatorname{ker}}}

\newcommand\SL{\operatorname{SL}}
\newcommand\Sp{\operatorname{Sp}}
\newcommand\Spin{\operatorname{Spin}}
\newcommand\SO{\operatorname{SO}}
\renewcommand\O{\operatorname{O}}
\newcommand\Spec{\operatorname{Spec}}

\newcommand\inverse{^{-1}}

\newcommand\rk{\operatorname{rk}}
\newcommand\Ker{\operatorname{ker}}

\numberwithin{equation}{section}

\theoremstyle{plain}
\newtheorem{lemma}[equation]{Lemma}
\newtheorem{theorem}[equation]{Theorem}
\newtheorem*{Theorem}{Theorem}

\newtheorem{corollary}[equation]{Corollary}
\newtheorem{proposition}[equation]{Proposition}
\theoremstyle{definition}

\newtheorem{remark}[equation]{Remark}


\begin{document}

\title[Spherical subgroups in simple algebraic groups]
{Spherical subgroups in simple algebraic groups}

\dedicatory{In memory of Tonny Springer}

\author[F. Knop]{Friedrich Knop}
\email{friedrich.knop@fau.de}
\address
{Department Mathematik,
Emmy-Noether-Zentrum,
FAU Erlangen-Nürnberg,
Cauerstrasse 11,
91058 Erlangen, Germany}

\author[G. Röhrle]{Gerhard Röhrle}
\email{gerhard.roehrle@rub.de}
\address
{Fakultät für Mathematik,
Ruhr-Universität Bochum,
D-44780 Bochum, Germany}


\allowdisplaybreaks

\begin{abstract}
  Let $G$ be a simple algebraic group. A closed subgroup $H$ of $G$ is
  called spherical provided it has a dense orbit on the flag variety
  $G/B$ of $G$.  Reductive spherical subgroups of simple Lie groups were
  classified by Krämer in 1979.  In 1997, Brundan showed that each
  example from Krämer's list also gives rise to a spherical subgroup
  in the corresponding simple algebraic group in any positive
  characteristic.  Nevertheless, there is no classification of all
  such instances in positive characteristic to date.  The goal of this
  paper is to complete this classification.  It turns out that 
  there is only one additional instance (up to isogeny) in characteristic $2$ 
  which has no counterpart in Krämer's classification.

As one of our key tools, 
we prove a general deformation result for subgroup schemes allowing us to 
deduce the sphericality of subgroups in positive characteristic 
from this property for subgroups in characteristic 0. 
\end{abstract}

\maketitle


\section{Introduction}

Let $G$ be a simple algebraic group defined over an algebraically
closed field $k$ of characteristic $p\ge0$.  A closed subgroup $H$ of
$G$ is called \emph{spherical} provided it has a dense orbit on the
flag variety $G/B$ of $G$.  Alternatively, $B$ acts on $G/H$ with an
open dense orbit.  Accordingly, a $G$-variety with this property is
also called \emph{spherical}.

The purpose of this paper is to classify connected reductive spherical
subgroups of simple groups in arbitrary characteristic. Thereby, we
generalize Krämer's classification \cite{Kr} in characteristic zero.

The class of reductive spherical subgroups is of particular
importance. This is shown by the fact that Krämer's list permeates
much of the theory of spherical varieties in characteristic zero. In
particular, these kind of subgroups provide many of the building
blocks for arbitrary spherical subgroups (see e.g.\ 
\cite{BrPezz}). We expect reductive spherical subgroups to play a
similar role for arbitrary $p$. In fact, the results of the present
paper were already used in \cite{Knop5} to list all spherical
subgroups of rank $1$, which is crucial for the theory of general
spherical varieties.

For $p\ne2$, the class of reductive spherical subgroups includes all
symmetric subgroups, i.e., subgroups which are fixed points of an
involutory automorphism of $G$ (see e.g.\ Springer \cite{springer}). On
the other side, for $p=2$ symmetric subgroups are not well behaved at
all. Thus, reductive spherical subgroups seem to be the correct 
replacement.

Note that the requirement of having an open orbit in $G/B$ entails
that $H$ has in fact only finitely many orbits (see
e.g.\ \cite{Knop3}). Therefore, our classification theorem can also be
viewed as a contribution
to the program by Seitz \cite{Seitz} to classify all pairs of subgroups 
$X,Y$ of a reductive group $G$ such that there are only finitely many
$(X,Y)$-double cosets in $G$, see also \cite{Br} and
\cite{Duck}.

The most important previous work is the aforementioned paper
\cite{Kr} by Krämer. Not only do we use his list as a guideline,
more importantly, it enters our computations crucially even in positive
characteristic. 
This is because we employ extensively the technique
of reduction mod $p$ which has been first used by Brundan \cite{Br} in
this context in 1997. 
There he shows that all items from Krämer's list descend to
arbitrary positive characteristic. To this end, Brundan proves,
\cite[Thm. 4.3]{Br}, that if $H$ and $G$ are defined over $\BBZ$ and
some additional 
technical condition holds, then $H$ is spherical for any $p>0$ if
and only if $H$ is spherical for $p=0$. In the present paper, we remove
the technical condition in Theorem \ref{Theorem1}, making the 
reduction mod $p$ technique
much more flexible to use. In particular, we barely ever have to check
sphericality of a given subgroup. Instead we just have to look it up in
Krämer's list.

Note that for the
purpose of classifying spherical subgroups, we may replace $G$ with
an isogenous group (using Lemma~\ref{lem:isogeny}). Therefore, 
the simply connected Spin groups do not make an appearance
in Table~\ref{table:sp-classical}, 
for instance, but rather their isogenous counterparts do.

We now describe our results in detail. The only surprise is that, up to isogeny,
there is only one case, namely in characteristic $2$, which is really
genuine to positive characteristic, i.e., which has no counterpart in
Krämer's classification.

\begin{Theorem}
  Let $G$ be a simple algebraic group and let $H \subset G$ be a
  closed connected reductive subgroup of $G$.  Then $H$ is spherical
  in $G$ if and only if $H$ is one of the groups in
  Table~\ref{table:sp-classical} ($G$ classical) or
  Table~\ref{table:sp-exceptional} ($G$ exceptional).
\end{Theorem}

Our classification is actually a bit more comprehensive, since we
classify the connected reductive spherical subgroups of all classical
groups not only up to outer but also up to inner automorphisms of
$G$. Here, by a \emph{classical group} we mean one of the groups
$\SL(n)$ ($n\ge2$), $\SO(n)$ ($n\ge1$), and $\Sp(n)$ ($n\ge2$ even)
which comprises also the non-simple groups $\SO(2)$ and $\SO(4)$. In
positive characteristic, the latter group contains infinitely many
``new'' spherical subgroups namely the images of $\Delta_q$, where
$\Delta_q:\SL(2)\to\GL(4)$ denotes the irreducible representation of
$\SL(2)$ with highest weight $(q+1)\omega_1$ with $q=p^m>1$. Since
$\Delta_q$ is selfdual, its image lies in $\SO(4)$.
We note that the images of $\Delta_q$ are special 
cases of finite orbit modules involving Frobenius twists, 
cf.~\cite[Lem.\ 2.6]{GLMS}.

\begin{table}[h!t]
\renewcommand{\arraystretch}{1.5}
\[
  \begin{array}{ll|lll}
\multicolumn{2}{l}{\text{Cases for all $p \ge 0$}} & 
\multicolumn{2}{l}{\text{Additional cases for $p > 0$}} \\
\hline\hline
H & G & H & G \\
\hline
\SO(n)^{(1)} & \SL(n)\ (n\ge2)\\
S(\GL(m){\times}\GL(n)) & \SL(m{+}n)\ (m\ge n\ge1) \\
\SL(m)\times\SL(n)&\SL(m{+}n)\ (m>n\ge1)\\
\Sp(2n)&\SL(2n)\ (n\ge2)\\
\Gm\cdot\Sp(2n)&\SL(2n+1)\ (n\ge1)\\
\Sp(2n)&\SL(2n+1)\ (n\ge1)\\
\hline
\GL(n)&\Sp(2n)\ (n\ge1)\\
\Gm\times\Sp(2n-2)&\Sp(2n)\ (n\ge2)\\
\Sp(m){\times}\Sp(n)&\Sp(m{+}n) &
\SO(m){\times}\SO(n)&\SO(m{+}n{-}1)& p=2  \\ 
\multicolumn{2}{l}{\quad\text{$m$, $n\ge2$ even}} \vline
&\multicolumn{2}{l}{\quad\text{$m$, $n\ge3$ odd}}\\
& &\G_2\times\SO(3)&\SO(9)& p=2 \\
\hline                                
\GL(n)^{(2)}&\SO(2n)\ (n\ge2)\\
\SL(n)&\SO(2n)\ (n\ge3\text{ odd})\\
\Sp(4)\otimes\Sp(2)^{(3) (4)}&\SO(8)\\
\Spin(7)^{(3) (5)}&\SO(8)\\
\G_2&\SO(8)\\
\SO(2)\times\Spin(7)&\SO(10)\\
\GL(n)&\SO(2n+1)\ (n\ge2)\\
\SO(m){\times}\SO(n)&\SO(m{+}n)& \SO(2m){\times}\Sp(2n) & \Sp(2m{+}2n)& p=2  \\
\quad m\ge n\ge1&&\multicolumn{2}{l}{\quad m\ge1, n\ge0}\\
\Spin(7)&\SO(9)&\Spin(7)&\Sp(8) & p=2 \\
\G_2&\SO(7)&\G_2&\Sp(6) & p=2 \\
& &\G_2\times\Sp(2)&\Sp(8) & p=2 \\
\hline
&&\Delta_q\SL(2)^{(3)} &\SO(4)&q>1\\
\hline                                
\multicolumn{5}{l}{\text{\Small
${}^{(1)}$ For $p=2$ and $n\ge3$ odd there are two 
classes 
which are swapped by an outer 
automorphism of~$G$.}}\\[-5pt]

\multicolumn{5}{l}{\text{\Small
${}^{(2)}$
For $n$ even there are two classes 
which are swapped by an outer 
automorphism  of~$G$.}}\\[-5pt]

\multicolumn{5}{l}{\text{\Small
$^{(3)}$ 
There are two conjugacy classes of $H$ in $G$ which are
swapped by an outer automorphism of~$G$.}}\\[-5pt]

\multicolumn{5}{l}{\text{\Small
$^{(4)}$ 
Using triality, $H\subset G$
is equivalent to $\SO(5)\times\SO(3)\subset\SO(8)$.}}\\[-5pt]

\multicolumn{5}{l}{\text{\Small
$^{(5)}$ 
Using triality, $H\subset G$
is equivalent to $\SO(7)\subset\SO(8)$.}}

 \end{array}  
\]

\smallskip
\caption{Spherical pairs $H\subset G$ with $G$ classical.} 
\label{table:sp-classical} 
\end{table}

\begin{table}[h!t]
\renewcommand{\arraystretch}{1.4}
  \begin{equation*}
  \begin{array}{ll|lll}
\multicolumn{2}{l}{\text{Cases for all $p \ge0$}} & 
\multicolumn{3}{l}{\text{Additional cases for $p > 0$}} \\ 
\hline\hline
H & G & H & G \\
\hline
\A_2&\G_2&\widetilde\A_2&\G_2& p=3\\
\A_1\times\widetilde\A_1&\G_2\\
\B_4&\F_4&\C_4&\F_4& p=2\\
\C_3\times\A_1&\F_4&\B_3\times\widetilde\A_1&\F_4& p=2 \\
\C_4&\E_6\\
\F_4&\E_6\\
\D_5&\E_6\\
\Gm\cdot\D_5&\E_6\\
\A_5\times\A_1&\E_6\\
\Gm\cdot\E_6&\E_7\\
\A_7&\E_7\\
\D_6\times\A_1&\E_7\\
\D_8&\E_8\\
\E_7\times\A_1&\E_8\\
\hline                                
 \end{array}  
  \end{equation*}\smallskip
\caption{Spherical pairs $H\subset G$ with $G$ exceptional.} 
\label{table:sp-exceptional} 
\end{table}


Note that the left-hand sides of 
Tables \ref{table:sp-classical} and \ref{table:sp-exceptional}
just reproduce Krämer's results. The cases on the right hand 
are new in positive characteristic. 
They are arranged in such a way that the case on the
right can be obtained from the one on the left 
by a non-central
isogeny of $G$. Thus, the only new case which has no counterpart in 
Krämer's table is the following
\begin{equation*}
\label{eq:1}
H=\G_2\times\Sp(2)\subset\Sp(6)\times\Sp(2)\subset G=\Sp(8)
\end{equation*}
for $p = 2$. Of course, there is also
$\G_2\times\SO(3)\subset\SO(9)$ which is isogenous to this case.

\vfill
\eject

In Table \ref{table:sp-exceptional},
$\widetilde\A_1$ and $\widetilde\A_2$ refer to a subgroup of $G$ of type
$\A_1$ and $\A_2$, respectively, whose root system only consists of short roots.

\section{Preliminaries}
\label{sec:basics}

\subsection{Notation}
Throughout, $G$ is a simple algebraic group 
and $B$ denotes a Borel subgroup of $G$.
By $\rk G$ we denote the rank of $G$.
Let $H$ be a closed subgroup of $G$.
Then $R_u(H)$ denotes the unipotent radical of $H$. 
If $G$ acts on the variety $X$, we denote the $H$-orbit of $x$ in $X$ by
$H \cdot x$ and the stabilizer in $H$ by $C_H(x)$.

In the sequel 
we use the labeling of the Dynkin diagram of a simple group $G$
according to the tables in Bourbaki \cite{bourbaki:groupes},
and $\omega_i$ denotes the $i$-th fundamental dominant weight 
of $G$ with respect to this labeling. 

For $\chi$ a dominant weight of $G$, we denote by $L(\chi)$ 
the irreducible $G$-module of highest weight $\chi$
and by $H^0(\chi)$ the $G$-module of global sections of
the $G$-line bundle $\CL(k_\chi)$ on $G/B$ afforded by the weight $\chi$,
so that $L(\chi) = \soc_G H^0(\chi)$. 
Note that $H^0(\chi)$ 
has the same character as the Weyl module of 
highest weight $\chi$; for details, see \cite[II.2]{jantzen}.

By a \emph{classical group} we mean one of the groups $\SL(n)$
($n\ge2$), $\SO(n)$ ($n\ge1$), or $\Sp(n)$ ($n\ge2$ even).
Here, $\SO(n)$ is the reduced, connected identity component of
$\O(n)$, i.e., the kernel of the determinant character $\det$ unless
$p=2$ and $n$ is even where $\det$ has to be replaced by the Dickson
invariant.

\subsection{Basic Results 
for Spherical Subgroups}
While elementary, one of our main tools 
to identify spherical subgroups 
(apart from Theorem \ref{Theorem1} below)
is the following necessary condition.

\begin{lemma} 
\label{lemma3}
Let $H\subseteq G$ be spherical in $G$. Then
\begin{equation}
\label{eq:2}
\dim H\ge\dim G/B=\frac{1}{2}(\dim G-\rk G).
\end{equation}
\end{lemma}

\begin{proof}
By definition, $B$ has an open
orbit in $G/H$. Hence $\dim B\ge\dim G/H$ which is equivalent to
\eqref{eq:2}.
\end{proof}

In the sequel we 
use the following ``transitivity'' property for 
spherical subgroups without further comment.

\begin{lemma} 
\label{lem:subgroups}
Let $H_1\subseteq H_2\subseteq G$ be connected 
reductive subgroups of $G$.
If  $H_1$ is spherical in $G$,
then $H_1$ is spherical in $H_2$ and  $H_2$ is spherical in $G$.
\end{lemma} 

\begin{proof}
Suppose that $H_1$ is spherical in $G$.
Then $H_1$ acts on $G/B$ with a dense orbit
and so does $H_2$ and thus $H_2$ is spherical in $G$. 

Let $B_2 \subset H_2$ be a Borel subgroup of $H_2$.
Then there is a Borel subgroup $B$ of $G$ such that $B_2 = H_2 \cap B$, 
e.g.\ see \cite[Cor.\ 2.5]{BMR}.
Consider the canonical embedding $H_2/B_2 \to G/B$.
Thanks to the finiteness result for irreducible, spherical $G$-varieties in 
arbitrary characteristic, \cite[Cor.\ 2.6]{Knop3}, 
since $H_1$ is spherical in $G$, 
$H_1$ admits only a finite number of orbits in $G/B$.
Thus there is only a finite number of $H_1$-orbits in
$H_2/B_2$ and in particular, there is a dense one.
Consequently, $H_1$ is spherical in $H_2$.
\end{proof}

The following compatibility of sphericality for
direct products is immediate from the definition of a spherical subgroup
and is also used in the sequel without further reference.

\begin{lemma} 
\label{lem:products}
Let $H_i\subseteq G_i$ be a 
reductive subgroup of $G_i$
for $i = 1,2$.
Then $H_1 \times H_2$ is spherical in $G_1 \times G_2$ if and only if
$H_i$ is spherical in $G_i$ for both $i =1,2$.
\end{lemma} 

Sometimes the following stronger bound on $\dim H$ is needed
in place of the inequality \eqref{eq:2}. 

\begin{lemma} 
\label{lemma3refined}
Let $H\subseteq G$ be spherical with 
$B\cdot x_0\in G/H$ the open $B$-orbit in $G/H$. Then
\begin{equation}
\label{eq:2refined}
\dim H=\dim G/B+\dim C_B(x_0).
\end{equation}
\end{lemma}

\begin{proof}
This follows, as $\dim B - \dim C_B(x_0) = \dim B\cdot x_0 = \dim  G/H
= \dim G - \dim H$.
\end{proof}

We also frequently use the following observation.

\begin{lemma}
\label{lem:isogeny}
  Let $G_1$ and $G_2$ be connected reductive groups and $\varphi:G_1\to
  G_2$ an isogeny.
Then $\varphi$ induces a bijection between the sets of
  (conjugacy classes of) connected (reductive) spherical subgroups of
  $G_1$ and $G_2$.
\end{lemma}

Lemma \ref{lem:isogeny}  has several immediate consequences.

\begin{remark}
\label{rem:isogeny}
(i) The triality automorphism of $\Spin(8)$ acts on the conjugacy
classes of connected reductive spherical subgroups of $\SO(8)$, as
well. This action is  indicated in Table~\ref{table:sp-classical}.

(ii) In characteristic $p=2$, there is a bijective non-central isogeny
$\SO(2n+1)\to\Sp(2n)$. Thus, if $G$ is a classical group we may (and
will) safely assume that $G$ is \emph{strictly classical} in the sense
that $G$ is not isomorphic to $\SO(2n+1)$ where $n\ge1$ when $p=2$.
Equivalently, a classical group is strictly classical if its natural
representation is completely reducible.
\end{remark}

\section{Deformation of Spherical Subgroups}
\label{sec:deformation}
In this section, we prove that ``sphericality'' is invariant under
deformations. This enables us to compare spherical subgroups in
positive characteristic to those in characteristic zero. This approach
reduces most of the classification work to Krämer's paper \cite{Kr}.

For simplicity, we restrict ourselves to base schemes $S$ which are of
the form $\Spec A$, where $A$ is a Dedekind domain\footnote{Our main
  assertion is surely valid in greater generality but due to technical
  difficulties stemming from the construction of coset schemes in
  \cite{Anan} and closures of subgroup schemes in \cite{bruhattits}
    we stick to Dedekind rings.}, i.e., an integrally closed
  Noetherian domain of dimension $1$.

In the sequel, let $\CG\to S$ be a split reductive group scheme (this
entails connected geometric fibers), e.g., see 
\cite[Exp.~I, 4.2]{demazurgrothendieck:schemasengroupes}.  
Let $\CT$ be a split maximal torus of $\CG$. 
Using \cite[Exp.~XXII, Cor.\ 5.5.5(iii)]{demazurgrothendieck:schemasengroupes}, 
a Borel subgroup scheme $\CB$ of $\CG$ containing $\CT$ has the
form $\CB_{R^+}$,  where $R^+$ is a system of positive roots for $\CG$. 
Then thanks to 
\cite[Exp.~XXII, Lem.\ 5.5.6(iii)]{demazurgrothendieck:schemasengroupes}, 
$\CB$ is the semidirect product $\CB = \CT \cdot \CU$ for a smooth
subgroup scheme $\CU$, and $\CU_k$ is a maximal connected unipotent subgroup
in $\CG_k$ for any $A$-algebra $k$ which is a field. 
Let $\Xi$
be the character group of $\CB$.  For $\CX\to S$ an affine $S$-scheme
let $\CO(\CX)$ be its ring of regular functions.

We need the following extension property for invariants due to
Seshadri \cite{Sesh}. See also \cite{FvdK} for a simplified approach.

\begin{lemma} 
\label{lemma1} 
Let $\CX\to  S$ be an affine $\CG$-scheme and
$\CY\subseteq\CX$ a closed $\CG$-invariant subscheme of $\CX$. Then
for any $\CG$-invariant function $f\in\CO(\CY)^\CG$ there is an
$n\ge1$ such that $f^n$ extends to a $\CG$-invariant function $\fbar$ on $\CX$.
\end{lemma}

Next we prove that the extension property from Lemma \ref{lemma1} 
also holds for $\CB$-semi-invariants.

\begin{lemma} 
\label{lemma2}
Let $\CX$ and $\CY$ be as in Lemma \ref{lemma1}. Let $f\in\CO(\CY)$ be a
$\CB$-semi-invariant function for a character $\chi\in\Xi$. Then there
is an exponent $n\ge1$ such that $f^n$ extends to a
$\CB$-semi-invariant function $\fbar\in\CO(\CX)$ for the character
$n\chi$.
\end{lemma}

\begin{proof} 
Let $\CG\mod\CU$ be the basic affine space of $\CG$; it is the
spectrum of $\oplus_{\chi\in\Xi}H^0(\chi)$.
(Note that $H^0(\chi)$ is a free $A$-module, thanks to 
\cite[II 8.8]{jantzen}.)
Thus $\CG\mod\CU$ is an affine
scheme over $S$ which contains $\CG/\CU$ as dense open subset. In
particular, $\CG\mod\CU$ contains an $S$-point $e$. The $\Xi$-grading
of $\CO(\CG\mod\CU)$ induces an action of $\CT=\CB/\CU$
which commutes with the $\CG$-action.

Consider the closed embedding $\CX  \xrightarrow{\id \times e}\CX\times_S\CG\mod\CU$. Then it is well-known (\cite{FvdK} proof of
Lemma 24) that restriction to $\CX$ induces a $\CT$-equivariant
isomorphism
$$
\CO(\CX\times_S\CG\mod\CU)^\CG \xrightarrow{\sim} \CO(\CX)^\CU.
$$
Thus our assertion follows from Lemma \ref{lemma1} applied to
$\CY\times_S\CG\mod\CU\subseteq\CX\times_S\CG\mod\CU$ and the fact
that $\CT$ is linearly reductive.
\end{proof}

\begin{remark}
\label{rem1}
If $\CY$ is actually defined over a prime field, say $\BBQ$ or
$\BBF_p$, then the exponent $n$ in Lemmas \ref{lemma1} and \ref{lemma2}
can be chosen to be $n=1$ or $n\in p^\BBN$, respectively.
\end{remark}

Now we are in the position to prove our main deformation statement:

\begin{theorem} 
\label{Theorem1} 
Let $\CH\subseteq\CG$ be a 
subgroup scheme which is flat over
$S$. Assume that for some geometric point $x$ of $S$,
the geometric fiber $\CH_x$ is a spherical subgroup of $\CG_x$. 
Then all geometric fibers
of $\CH$ are spherical.
\end{theorem}

\begin{proof} 
Since $S$ is the spectrum of a Dedekind ring,
the closure $\overline\CH$
of $\CH$ in $\CG$ is a flat closed subgroup scheme, as well
cf.\ \cite [1.2.6, 1.2.7, 2.1.6, 2.2.2]{bruhattits}. Moreover, $\CH_x$ is
spherical in $\CG_x$ if and only if $\overline\CH_x$ is (since the
former is open, hence of finite index in the latter). Thus, after
replacing $\CH$ by $\overline\CH$, we may assume that $\CH$ is closed
in $\CG$.

In that case, it is known that the homogeneous space $\CX':=\CG/\CH$
exists as a scheme, which is flat and of finite type over $S$ (see
\cite{Anan}).  Moreover, by Sumihiro (\cite{Sum}, see also
\cite{Thom}), this scheme is equivariantly quasiprojective over
$S$. This means that there is a $\CG$-vector bundle $\CV$ over $S$
and an equivariant embedding of $\CX'$ in the projective space
$\BBP_S(\CV)$. Let $\CX''\subseteq\BBP_S(\CV)$ be the closure of
$\CX'$. This is a scheme which is projective and flat over
$S$. Moreover, each geometric fiber $\CX'_x=\CG_x/\CH_x$ is an open
subset of the fiber $\CX''_x$.

Now let $\CX\subseteq \A_S(\CV):=\Spec S^\bullet\CV$ be the affine cone
of $\CX''$. The affine scheme $\CX$ affords an action of
$\CGtilde:=\CG\times_S(\Gm)_S$. Moreover, an irreducible
subvariety of $\CX''_x$ is spherical as a $\CG_x$-variety if and only
if its affine cone in $\CX_x$ is a spherical
$\CGtilde_x$-variety. Thus, by replacing $\CG$ by $\CGtilde$ we may assume that 
$\CX'=\CG/\CH$ is an open dense subscheme of an affine scheme $\CX$.

Suppose that $\CX_x$ has a spherical irreducible component.
Let $y\in S$ be a second geometric point. We have to prove that every component
of $\CX_y$ is spherical as well. Let $\eta$ be the generic geometric point of
$S$. 
First, we are going to show that $\CX_\eta$ and, subsequently, that $\CX_y$
is spherical. This amounts to the same as assuming that either $y=\eta$
or $x=\eta$.

Assume first that $y=\eta$. Let $Y\subseteq\CX_x$ be a spherical
irreducible component. This means that somewhere on $Y$  
the dimension of the isotropy
subgroup of $\CB$ is as small as possible, namely
$\dim \CB_x-\dim \CX_x=\dim \CB-\dim \CX$. By semi-continuity, this
holds on a non-empty open subset $\CX^0$ of $\CX$. Since then
$\CX^0\cap\CX_\eta\ne\varnothing$, we conclude that $\CX_\eta$ is spherical
(observe for this that $\CX_\eta$ is irreducible since it contains an open dense $\CG_\eta$-orbit).

Finally, let $x=\eta$ and suppose that some component $Y$ of $\CX_y$
is not spherical. Then, by Rosenlicht \cite{Rosenlicht}, 
$Y$ admits a non-constant
rational $\CB_y$-invariant function $f$. Because $Y$ is affine, this
function can be written as $f=f_1/f_2$, where $f_1,f_2\in\CO(Y)$ are
$\CB_y$-semi-invariants for the same character $\chi\in\Xi$. By
Lemma \ref{lemma2}, 
there is an $n\in\BBN$ such that $f_1^n$ and $f_2^n$ extend to
$\CB$-semi-invariants $\fbar_1,\fbar_2$ for the same character $n\chi$ on all of $\CX$. Now,
since $\CX$ is integral, we obtain a $\CB$-invariant rational function
$\fbar=\fbar_1/\fbar_2$ on $\CX$ which is not a constant, i.e., a pull-back
of a function on $S$. Thus, in particular, the generic fiber $\CX_x$
is not spherical contrary to our assumption.
\end{proof}

Applying Theorem \ref{Theorem1} to $S=\Spec \BBZ$, we get the following result,
which has been previously obtained by Brundan \cite[Thm.\ 4.3]{Br}, 
using a representation-theoretic approach and 
partially based on case-by-case considerations.

\begin{corollary}
\label{cor:lift}
Let $H_\BBR\subseteq G_\BBR$ be a pair of compact Lie groups in
Krämer's list, i.e., with $H_\BBR$ spherical in $G_\BBR$. Then the complexification $H_\BBC\subseteq G_\BBC$ has
a $\BBZ$-form $H_\BBZ\subseteq G_\BBZ$. Moreover, for any field
$k$ the induced pair $H_k\subseteq G_k$ is spherical.
\end{corollary}

\begin{proof}
The first statement follows by inspection of Krämer's list. 
The second follows from the first together with Theorem \ref{Theorem1} 
for $S=\Spec \BBZ$.
\end{proof}

In the reverse 
direction, we recover a classification of Duckworth \cite[Thm.\ 2]{Duck}
which can be formulated as follows.

\begin{corollary}
Assume that $p\ne2$ if $G$ is of type $\B_n$, $\C_n$, or $\F_4$ 
and that $p\ne3$ if $G$ is of type $\G_2$.
Then the classification of pairs $(G, H)$,  
where $G$ is a simple group
and $H$ is a spherical subgroup of $G$ with $\rk H=\rk G$, 
is independent of $p$.
\end{corollary}

\begin{proof}
Under the given restrictions, 
$H$ corresponds to an additively closed
subroot system. Therefore, it lifts to characteristic
$0$.  Then apply Theorem \ref{Theorem1} for $S=\Spec \BBZ$.
\end{proof}

\section{Special Cases of Spherical Subgroups}
\label{sec:special}

For an arbitrary $G$-variety $X$ let $\Xi(X)$ be the group of
characters of $B$-semi-invariant rational functions on $X$. We denote
the rank of $X$ as the $\BBZ$-rank of $\Xi(X)$.  Let $S_0$ be the set
of simple roots $\alpha$ such that the coroot $\alpha^\vee$ is
orthogonal to $\Xi(X)$. Then attached to $S_0$ is a parabolic subgroup
$P = P(X)$ of $G$ such that $\Xi(X)\subseteq \Xi(P)$, where
$\Xi(P)$ is the character group of $P$.  We define the
subgroup $P_0$ of $P$ by $P_0 = \{y \in P \mid \chi(y)=1\text{ for all }
\chi\in\Xi(X)\}$.

\begin{theorem}
\label{thm:genstab}
Let $X$ be a quasiaffine $G$-variety. Let $P = P(X)$ as above. Then
there is a $P$-invariant dense open subset $X_0$ of $X$ such that
$C_P(x) R_u(P) = P_0$ and $C_P(x)\cap R_u(P)$ is finite for all $x\in
X_0$. In particular, $C_P(x)$ is a reductive group which is isogenous
to a Levi subgroup of $P_0$.
\end{theorem}

\begin{proof}
According to \cite[Satz 2.10]{Knop}, there is a parabolic subgroup
$P$ of $G$ and a $P$-stable dense open subset $X_0\subseteq X$
such that:
\begin{itemize}
\item[(i)] The action of $R_u(P)$ on $X_0$ is proper.
\item[(ii)]The orbit space $Y:=X_0/R_u(P)$ exists.
\item[(iii)] Let $P_1$ be the kernel of the $P$-action on $Y$. 
Then $P/P_1$ is a torus.
\item[(vi)] The action of $P/P_1$ on $Y$ is free.
\end{itemize}
Let $\pi:X_0\to Y$ be the quotient map and $x\in X_0$. Then $C_P(x)\cap
R_u(P)$ is finite by 1. We have $C_P(x)R_u(P)\subseteq C_P(y)$ with
$y=\pi(x)$. Moreover, for $z\in C_P(y)$ there is $u\in R_u(P)$ with
$zx=ux$. Thus $C_P(x)R_u(P)=C_P(y)$. Finally, $C_P(y)=P_0$ by (iv).

It remains to show that $P = P(X)$ and $P_1 = P_0$, as defined
above. For this we use the fact that $X_0$ is constructed as the
non-vanishing set of a $B$-semi-invariant section $\sigma$ of a
sufficiently high power $\CL^n$ of any ample line bundle on $X$. Since
$X$ is quasiaffine, we can take $\CL=\CO_X$. Moreover, $P$ is the
stabilizer of the line $k\sigma$. Since $\sigma$ is a regular function
on $X$, the $G$-module $M:=\langle G\sigma\rangle_k $ generated by $\sigma$ is
finite-dimensional and $\sigma$ is a highest weight vector in $M$ with
weight denoted by $\chi$. This implies that $P$ is the parabolic subgroup
attached to the set $S_1$ of simple roots $\alpha$ with
$\langle \chi,\alpha^\vee\rangle =0$. From the construction it is easy to see that
$\chi$ can be chosen such that $S_1=S_0$. This shows that indeed  $P = P(X)$.
Finally, observe that
$\Xi(X)=\Xi(Y)$. Thus properties (iii) and (iv) ensure that $P_1$ is the common
kernel of all $\chi\in\Xi(X)$, i.e.\ $P_1 = P_0$.
\end{proof}

\begin{lemma}
\label{lemma:SOSP}
Let $X=\SO(n)/\SO(n-m)$ or $X=\Sp(2n)/\Sp(2n-2m)$ with $2m\le n$. 
Then $\Xi(X)\subseteq\langle \omega_1,\ldots,\omega_{2m}\rangle_\BBZ$.
\end{lemma}

\begin{proof}
Let $G =\SO(n)$ and $H = \SO(n-m)$ or
$G = \Sp(2n)$ and $H = \Sp(2n-2m)$ with $2m\le n$,
respectively. 
Write $X=G/H$. 
First observe that $X$ lifts to characteristic zero,
thanks to Corollary \ref{cor:lift}.
 Since the
character group $\Xi(X)$ is the same in
characteristic $0$ and in positive characteristic $p$ (after inversion of $p$), 
we may assume from the outset that $\Char k = 0$.

Since $X$ is affine, every rational $B$-semi-invariant
is the ratio of two regular ones. Moreover, a regular
$B$-semi-invariant with character $\chi$ corresponds to a non-zero
$H$-fixed vector in the dual irreducible $H$-module $L(\chi)^*$. Now it
follows readily from classical branching laws (e.g., see 
\cite[Ch.\ 8]{GW}) that $\chi$ is a linear combination of the first $2m$
fundamental weights.
\end{proof}

\begin{lemma}
\label{lemma:SOSP2}
\begin{itemize}
\item[(i)]
Let $H\subset\SO(m)$ be a proper, reductive subgroup of $\SO(m)$
such that $H\times\SO(n-m)$ is spherical in $\SO(n)$. Then $2m>n$.

\item[(ii)]
Let $H\subseteq\Sp(2m)$ be a reductive subgroup such that $H\times
\Sp(2n-2m)$ is spherical in $\Sp(2n)$. Assume that $2m\le n$. Then
$\dim H\ge\dim\SO(2m)=m(2m-1)$.
\end{itemize}
\end{lemma}

\begin{proof}
  (i) Suppose $\Htilde :=H\times\SO(n-m)$ is spherical in $\SO(n)$ and
  $2m\le n$.  Let $x_0\in G/\Htilde$ be in the open $B$-orbit in
  $G/\Htilde$. Then by \eqref{eq:2refined}, we have
  \[
  \dim \Htilde =\dim G/B+\dim C_B(x_0).
  \]
  By Theorem \ref{thm:genstab} and Lemma \ref{lemma:SOSP}, the generic
  isotropy group of $B$ on $\SO(n)/\SO(n-m)$ contains a subgroup which
  is isogenous to a Borel subgroup, say $B_2$, of $\SO(n-2m)$. Thus
  Lemma \ref{lemma3refined} implies $\dim C_B(x_0)\ge\dim B_2$.  To
  keep the dependence on the parity of $n$ to a minimum, observe that
  $\dim \SO(n)=\frac12n(n-1)$ and $\rk \SO(n)-\rk \SO(n-2m)=m$ for all
  $n$ and $m$. Hence we arrive at the following contradiction:
  \begin{align*}
    \dim H \ge&\ \frac12(\dim \SO(n)-\rk \SO(n))+
    \frac12(\dim \SO(n - 2m)+\rk \SO(n - 2m))\\
    & \quad - \dim \SO(n - m)\\
    =&\ \frac12m(m-1)=\dim \SO(m).
  \end{align*}
  For (ii) we argue in the same way and get
  \[
  \dim H\ge n^2+(n-2m)(n-2m+1)-(n-m)(2n-2m+1)=m(2m-1).
  \]
\end{proof}

\begin{corollary}
  \label{cor:spin7}
  Let $p=2$ and $n\ge5$. Then
  $H=\Spin(7)\times\Sp(2n-8)\subset\Sp(8)\times\Sp(2n-8)$ is not
  spherical in $G=\Sp(2n)$.
\end{corollary}

\begin{proof}
  For $n=5$, $6$, and $7$, the result follows from \eqref{eq:2}.  Now
  let $n\ge 8$. Noting that $21=\dim\Spin(7)<\dim \SO(8)=28$, it
  follows from Lemma \ref{lemma:SOSP2}(ii) that $H$ is not spherical.
\end{proof}

\begin{proposition}
  \label{prop:g2}

  Let $p=2$ and $n\ge4$. Then
  $H:=\G_2\times\Sp(2n-6)\subset\Sp(6)\times\Sp(2n-6)$ is spherical in
  $G=\Sp(2n)$ if and only if $n=4$. In that case $\Xi(G/H)=\langle
  \omega_1+\omega_4,\omega_2,\omega_3\rangle_\BBZ$.

\end{proposition}

\begin{proof}

  For $n=5$, the result follows from \eqref{eq:2}.  Let $n\ge6$. Since
  $\dim\G_2=14<\dim\SO(6)=15$, the assertion follows from Lemma
  \ref{lemma:SOSP2}(ii).

  It remains to check that $H$ is spherical if $n=4$. Put $\Htilde
  :=\Sp(6)\times\Sp(2)$ and write $\Htilde^0(\chi)$ for the
  corresponding $\Htilde$-module and
  $\Ltilde(\chi)$ for the simple $\Htilde$-module of highest weight
  $\chi$.
 
  We first show that
  $A:=\{\omega_1+\omega_4,\omega_2,\omega_3\}\subseteq\Xi(G/H)$. This
  is equivalent to the  $G$-modules $H^0(\chi)$ with $\chi\in
  A$ having an $H$-fixed vector.

  For $\chi=\omega_2$ this follows from the fact that even $\Htilde$
  has a fixed vector. Moreover, it is known that $\G_2$ fixes a vector
  in the $\Htilde$-module $\Htilde^0(\omega_3)$ which in turn is
  contained in $H^0(\omega_3)$.

  For $\chi=\omega_1+\omega_4$ it suffices to show that the
  irreducible $G$-module $L(\chi)\subset H^0(\chi)$ contains the
  $\Htilde$-module $\Htilde^0(\omega_3)$. Using the known characters
  of Weyl modules and the dimensions of the irreducible modules in
  \cite{Lue}, one easily computes that, as an $\Htilde$-module, 
  $L(\chi)$ has the composition factors
  $\Ltilde(\omega_1+\omega_2+\omega_1')$, $\Ltilde(2\omega_3+2\omega_1')$
  and $\Ltilde(\omega_3)$, the first two occur with multiplicity $1$ and the
  third one with multiplicity $2$. Since $L(\chi)$ is selfdual, (at least)
  one of the two copies of $\Ltilde(\omega_3)$ has to appear in the socle. This
  concludes the proof that $H^0(\omega_1+\omega_4)^H\ne \{0\}$.

  Since there is no simple coroot which is orthogonal to all the
  weights in $A$, we infer from Theorem \ref{thm:genstab} that the
  connected $B$-isotropy group of a generic point $x\in G/H$ is a
  torus of dimension at most one. Thus $\dim B \cdot x\ge 19$, whereas
  $\dim G/H=36-14-3=19$. This shows that $G/H$ is spherical of rank
  $3$. In particular, $\Xi(G/H)$ is spanned by $A$, as claimed.
\end{proof}

\section{Irreducible Spherical Subgroups of Classical Groups}
\label{sec:irreducible}

Let $G$ be a classical group with natural representation $V$. A
subgroup $H\subseteq G$ is called \emph{irreducible} if $V$ is
irreducible as an $H$-module. Otherwise, $H$ is called
\emph{reducible}. Clearly, irreducible subgroups only exist if $G$
itself is irreducible, i.e., strictly classical and not equal to
$\SO(2)$. It is well known that connected irreducible subgroups are
necessarily semisimple.

In preparation for determining the non-simple irreducible spherical
subgroups, we consider some very special cases:

\begin{lemma} 
\label{lemma4}
Of the following pairs $H \subset G$,
\begin{align*}
&\SL(m)\otimes \SL(n)\subset \SL(mn),\quad m\ge n\ge2,\\
&\SO(m)\otimes \SO(n)\subset \SO(mn),\quad m\ge n\ge2,\ m,n\ \text{even if } p=2,\\
&\Sp(m)\otimes \Sp(n)\subseteq \SO(mn),\quad m\ge n\ge2,\ m,n\ \text{even},\\
&\Sp(m)\otimes \SO(n)\subset \Sp(mn),\quad m,n\ge2,\ m\ \text{even},\ n\ \text{even if } p=2,
\end{align*}
only the following are spherical:
\begin{align*}
&\SL(2)\otimes \SL(2) \subset \SL(4),\\
&\SO(2)\otimes \SO(2) \subset \SO(4),\\
&\Sp(2)\otimes \Sp(2) =\SO(4),\\
&\Sp(4)\otimes \Sp(2) \subset \SO(8),\\
&\Sp(2)\otimes \SO(2) \subset \Sp(4).
\end{align*}
\end{lemma}

\begin{proof}
There are two possible proofs. First, observe that all
subgroups lift to characteristic $0$. Hence, the assertion follows (apart from the
trivial case $G=\SO(4)$) from 
 Corollary \ref{cor:lift} and Krämer's
classification \cite{Kr}. The second proof consists in directly using
the inequality \eqref{eq:2} which is easy and left to the reader.
\end{proof}
 Next we determine the irreducible, spherical
subgroups which are not simple.

\begin{lemma} 
\label{lemma5}
Let $G$ be a classical group and $H\subset G$ a proper,
connected, irreducible, spherical subgroup which is not simple. Then
the pair $H\subset G$ is one of the following:
\begin{align*}
\SO(4)&\subset \SL(4),\\ 
\Sp(4)\otimes \Sp(2)&\subset \SO(8),\\ 
\SO(4)&\subset \Sp(4),\qquad(\hbox{if }p=2).
\end{align*}
\end{lemma}

\begin{proof}
By assumption, there are decompositions $H=H_1\cdot H_2$ and
$V=V_1\otimes V_2$, where $V_i$ is an irreducible $H_i$-module. For
$G=\SL(n)$, Lemma \ref{lemma4} shows that $H_1,H_2\subseteq \SL(2)$,
which implies $H_1=H_2 = \SL(2)$, hence $H=\SO(4)$.

Now let $G=\SO(V)$ or $G=\Sp(V)$ and assume first that $p\ne2$. Since
$V=V_1\otimes V_2$ is selfdual, the same holds for the factors
$V_i$. Thus $H_i$ is either symplectic or orthogonal. Since $H\ne G$,
we have $G\ne\SO(4)$. Therefore, the only case to consider according
to Lemma \ref{lemma4} is $H_1\times H_2\subseteq
\Htilde:=\Sp(4)\times\Sp(2)$ and $G=\SO(8)$. But in that case $\dim
G/B=12$ while $\dim \Htilde=13$. This implies $H=\Htilde$, since a
semisimple group does not contain a reductive subgroup of codimension
one.

Now assume that $p=2$ and that $V$ is selfdual. Then each factor $V_i$ is still
selfdual and we claim that it is even symplectic. To show this let
$\beta: V_i\times V_i\to  k$ be a non-zero $H_i$-invariant
pairing. Schur's Lemma implies that $\beta$ is unique up to a
scalar. It is symmetric since otherwise
$\beta'(u,v)=\beta(u,v)+\beta(v,u)$ is non-zero and
symmetric. But then $\ell(v):=\sqrt{\beta(v,v)}$ is an $H_i$-invariant
linear form. The irreducibility of $V_i$ implies $\ell=0$. Thus
$\beta(v,v)\equiv0$ proving the claim.

Consequently, we have
\[
H\subseteq\Sp(V_1)\otimes\Sp(V_2)\subseteq\SO(V)\subset\Sp(V).
\]
According to Lemma \ref{lemma4}, we are left with the following cases. If
$G=\SO(V)$, then $H=\Sp(4)\otimes\Sp(2)$, as before. If $G=\Sp(V)$, then
$H$ is spherical in $\SO(V)$, as well. Thus either $H=\SO(4)\subset
G=\Sp(4)$ (which is spherical) or $H=\Sp(4)\otimes\Sp(2)\subset
G=\Sp(8)$ (which is not spherical, by \eqref{eq:2},  because $\dim G/B = 16$ and $\dim H = 13$).
\end{proof}

To determine all simple irreducible spherical subgroups, we need the
following estimate to bound the dimension of a simple $H$-module.
The proof of the result follows easily by inspection.

\begin{lemma}
\label{lemma6}
Let $H$ be a simple group with Weyl group $W_H$ and let $\omega$ be a
fundamental dominant weight of $H$ with
\[
|W_H\cdot \omega|\le2\sqrt{\dim H+{1/4}}+1.
\]
Then the pair $(H,\omega)$ appears in 
Table \eqref{eq:4}. 
\begin{table}[h!t]
  \renewcommand{\arraystretch}{1.2}
  \begin{equation}
    \label{eq:4}
    \begin{array}{ll}
      H                & \omega\\
      \hline\hline
      \A_1           &\omega_1\\ 
      \A_2           &\omega_1, \omega_2\\ 
      \A_3           &\omega_1, \omega_2, \omega_3\\ 
      \A_4           &\omega_1, \omega_2, \omega_3, \omega_4\\ 
      \A_n, n\ge5    &\omega_1, \omega_n\\
      \hline
      \B_3           &\omega_1, \omega_3\\ 
      \B_n, n\ge4 &\omega_1  
    \end{array}
    \qquad\qquad
    \begin{array}{ll}
      H                & \omega\\
      \hline\hline
      \C_2           &\omega_1, \omega_2\\ 
      \C_3           &\omega_1, \omega_3\\ 
      \C_n, n\ge4   &\omega_1\\ 
      \hline
      \D_4           &\omega_1, \omega_3, \omega_4\\ 
      \D_n, n\ge5   &\omega_1\\ 
      \hline
      \G_2             &\omega_1, \omega_2\\
      \ 
    \end{array}  
  \end{equation}
\end{table}
\end{lemma}

Now we determine the simple, irreducible, spherical subgroups $H$ of a
classical group $G$.

\begin{lemma}
\label{lemma7}
Let $G$ be classical and $H\subset G$ a proper, connected, irreducible,
spherical subgroup. Then, up to conjugacy in $G$, the pair $(G,H)$ 
appears in 
Table \eqref{eq:table1}. 
\begin{table}[h!t]
\renewcommand{\arraystretch}{1.3}
  \begin{equation}\label{eq:table1}
  \begin{array}{lllll}
H & \text{weight} & G && \text{conditions on } p\\
\hline\hline
\SO(n)^{(1)}&\omega_1&\SL(n) &  n\ge3& p\ne2 \text{ for } n\text{ odd}\\ 
\Sp(2n)     &\omega_1          &\SL(2n) &  n\ge2\\ 
\hline
\G_2        &\omega_1          &\SO(7)&    &p\ne2\\ 
\Delta_q\SL(2)^{(2)}&(q+1)\omega_1&\SO(4)&&q=p^m>1\\
\Spin(7)^{(2)}    &\omega_3          &\SO(8)\\ 
\Sp(4)\otimes \Sp(2)^{(2)}&\omega_1+\omega_1'&\SO(8)\\ 
\hline
\SO(2n)     &\omega_1          &\Sp(2n)  & n\ge2&p=2\\ 
\G_2       &\omega_1          &\Sp(6)    &   &p=2\\ 
\Spin(7)^{(1)}    &\omega_3          &\Sp(8)&      &p=2\\
\hline
\multicolumn{5}{l}{\text{\Small${}^{(1)}$ not maximal for $p=2$.}}\\[-5pt]
\multicolumn{5}{l}{\text{\Small${}^{(2)}$ there are two conjugacy classes.}}\\[-5pt]
\end{array}
\end{equation}
\end{table}
\end{lemma}

\begin{proof} 
In view of Lemma \ref{lemma5}, 
we may assume that $H$ is simple. Let
$\omega$ be the highest weight of $H$ in the defining representation
$V$ of $G$. If $p >0$, recall that $\omega$ is called
\emph{$p$-restricted} if $\langle\omega,\alpha^\vee\rangle<p$ for all simple roots
$\alpha$ of $H$. In any case, there is a unique expansion
\[
\omega=\sum_{i=0}^mp^i\omega^{(i)}\text{ with }\omega^{(m)}\ne0,
\]
where each $\omega^{(i)}$ is $p$-restricted. We may assume that
$\omega^{(0)}\ne0$, as well, since otherwise $H\to
G$ factors through a Frobenius morphism. Steinberg's Tensor Product
Theorem asserts that
\[
V=V_0\otimes\cdots\otimes V_m.
\]
where $V_i$ is simple with highest weight $p^i\omega^{(i)}$. If
$m\ge1$, then the embedding $H\to G$ factors through one of the
subgroups in Lemma \ref{lemma4}. It follows easily that $G=\SO(4)$ and
$H=\Delta_q\SL(2)$ for $q=p^m>1$.

Thus, we may assume from now on that $\omega=\omega^{(0)}$ is
$p$-restricted. The inequality \eqref{eq:2} implies the following
upper bounds on $\dim V$:
\begin{equation}
\label{table:3}
\dim V\le
\begin{cases}
\sqrt2\,\sqrt{\dim H+{1/8}}+{1\over2}&\text{if }G=\SL(n),\\
2\,\sqrt{\dim H}+1&\text{if }G=\SO(2n+1),\\
2\,\sqrt{\dim H+{1/4}}+1&\text{if }G=\SO(2n),\\
2\,\sqrt{\dim H}&\text{if }G=\Sp(2n).\\
\end{cases}
\end{equation}
Thus,
\begin{equation*}
\dim V\le2\,\sqrt{\dim H+{1/4}}+1
\end{equation*}
in all cases. Now we use the trivial dimension estimate
$\dim V \ge|W_H \cdot \omega|$ to conclude that $H$ is one of
the groups in Table \eqref{eq:4}
and $\omega$ is a linear combination of the
fundamental weights in the right hand column of the same table.

Assume first that $\omega$ is not a fundamental weight. Then we claim
that the inequalities in \eqref{table:3} leave only two cases to
consider, namely $(H,\omega)=(\A_1,2\omega_1)$ and
$(H,\omega)=(\A_1,3\omega_1)$.

For groups of small rank (${\rk}\,H\le4$ will do), this can be checked
using the tables of Lübeck~\cite{Lue}.  So let $\rk H\ge5$ and suppose
that $\omega$ is not a multiple of a fundamental weight. Then
according to Lemma \ref{lemma6}, $H=\A_n$ and
$\omega=a\omega_1+b\omega_n$ with $a,b\ge1$. In that case it is
readily checked that the Weyl group orbit of $\omega$ is too big.

Next, we consider the case where $\omega=a\omega_1$ with $2\le
a<p$. Then
\[
\omega':=\omega-\alpha_1=(a-2)\omega_1+b\omega_2
\]
is also a weight of $V$, where
$b=-\langle\alpha_1,\alpha_2^\vee\rangle >0$. But $\omega_2$ does not
occur in Table \eqref{eq:4} excluding this possibility. This finishes
the proof of the claim.

Finally, it remains to check whether the representations of $H$ with
highest weight $\omega$ define a proper spherical subgroup of a
classical group where $\omega$ is one of the fundamental weights of
Lemma \ref{lemma6} or one of the exceptional cases $(\A_1,2\omega_1)$
or $(\A_2,3\omega_1)$. To make it easier some remarks are in order:
Firstly, it clearly suffices to check the $\omega$ up to an
automorphism of $H$. Secondly, the representations $(\C_2,\omega_2)$
for $p=2$, $(\C_3,\omega_3)$ for $p=2$, and $(\G_2,\omega_2)$ for
$p=3$ factor through $(\C_2,\omega_1)$, $(\B_3,\omega_3)$, and
$(\G_2,\omega_1)$, respectively, and therefore they can be
omitted. The result is summarized in the following table:
\begin{table}[h!t]
\renewcommand{\arraystretch}{1.4}
\begin{equation*}
\label{eq:3}
  \begin{array}{lrccc}
H &\hspace{-10pt} weight & G=\SL &G=\SO  & G=\Sp \\
\hline\hline
\A_{n-1} \ (n\ge2)   & \omega_1 & = & -& -\\ 
\B_n \ (n\ge3,p\ne2 )&\omega_1 & \SO(2n{+}1)\subset\SL(2n{+}1) &= &-\\ 
\C_n \ (n\ge2)&\omega_1 &\Sp(2n) \subset\SL(2n) & - &=\\ 
\D_n \ (n\ge4)  & \omega_1 &\SO(2n) \subset\SL(2n) &= &\SO(2n) 
\mathop{\subset}\limits_{p=2}
\Sp(2n)\\ 
\A_1 \ (p\ne2)   &2\omega_1 &\SO(3) \subset\SL(3) &=&-\\ 
\A_1 \ (p\ne2,3) &3\omega_1 &\times &-&\times\\ 
\A_3&\omega_2 &\SO(6) \subset\SL(6) &=&\SO(6) 
\mathop{\subset}\limits_{p=2}
\Sp(6)\\ 
\A_4&\omega_2 &\times &- &-\\ 
\B_3&\omega_3 &\times &\Spin(7) \subset\SO(8)  & \Spin(7) 
\mathop{\subset}\limits_{p=2}
\Sp(8)\ \ \ \\ 
\C_2 \ (p\ne2) & \omega_2 &\SO(5) \subset\SL(5) & =&-\\ 
\C_3 \ (p\ne2) &\omega_3 &\times &- &\times\\ 
\G_2 & \omega_1 &\times & \ \ \ \ \G_2 \mathop{\subset}\limits_{p\ne2}\SO(7) & \ \ \ \  \G_2
\mathop{\subset}\limits_{p=2}
\Sp(6)\\ 
\G_2 \ (p\ne3) & \omega_2 & \times & \times & \times
\end{array}
\end{equation*}
\end{table}

Here the notation is as follows: ``$-$'' means that $H$ is not a subgroup
of $G$; ``$=$'' means that $H$ equals $G$, and  ``$\times$''
means that $H$ is not spherical in $G$ in all cases, because \eqref{eq:2} is
violated. 
\end{proof}

\section{$G$-Completely Reducible, Spherical 
Subgroups of Classical Groups}
\label{sec:cr}
Following Serre \cite{Ser}, we say that a subgroup $H$ of a reductive
group $G$ is {\it $G$-completely reducible} provided whenever $H$ is
contained in a parabolic subgroup $P$ of $G$, then it is contained in
a Levi subgroup of $P$.  Thanks to \cite[Prop.\ 4.1]{Ser}, a
$G$-completely reducible subgroup of $G$ is reductive.

Suppose $G$ is classical with natural module $V$.
Note that for $G = \SL(V)$ a subgroup $H$ of $G$ is $G$-completely reducible
if and only if $V$ is semisimple as an $H$-module, \cite[Ex.\ 3.2.2(a)]{Ser}.
If $p \ne 2$ then this also holds for 
$G = \SO(V)$ or $\Sp(V)$, \cite[Ex.\ 3.2.2(b)]{Ser}.
However, if $p = 2$, these two notions differ and for 
$G = \SO(V)$ or $\Sp(V)$ both implications may fail.
For example, if $p=2$, then $H=\SO(2n-1)$ is
$G$-completely reducible in $G=\SO(2n)$ (in fact, $H$ is not contained
in any proper parabolic subgroup) but $V$ is not a semisimple 
$H$-module for $n\ge2$.
See \cite[Ex.\ 3.45]{BMR}
for an example of a simple subgroup of $\Sp(V)$ 
which is semisimple on $V$ but not $G$-completely reducible
when $p = 2$.

In the following, we always assume that $G$ is strictly classical,
i.e., we exclude the case $G=\SO(n)$ when $p=2$ and $n$ is odd.

\begin{lemma}\label{lem:maxreducible}

  Let $G$ be a strictly classical group and $H\subset G$ be maximal
  among connected spherical, $G$-completely reducible subgroups. Then
  $H$ is contained in the following table:
  \renewcommand{\arraystretch}{1.3}
  \begin{equation}\label{eq:table2}
    \begin{array}{llll}
      H & G \\
      \hline\hline
      S(\GL(m)\times\GL(n))&\SL(m+n)&m\ge n\ge1\\ 
      \GL(n)&\SO(2n)&n\ge1\\ 
      \GL(n)&\Sp(2n)&n\ge1&p\ne2\\
      \Sp(2m)\times\Sp(2n)&\Sp(2m+2n)&m\ge n\ge1\\ 
      \SO(m)\times\SO(n)&\SO(m+n)&m\ge n\ge1&p\ne2\\
      \SO(2m)\times\SO(2n)&\SO(2m+2n)&m\ge n\ge1&p=2\\
      \SO(2n-1)&\SO(2n)&n\ge2&p=2
    \end{array}
  \end{equation}
\end{lemma}

\begin{proof}

  Let $\omega$ be the defining symplectic form of $\Sp(2n)$ and let
  $q$ be the defining quadratic form of $\SO(n)$.

  Choose a non-zero $H$-invariant subspace $U\subseteq V$ of minimal
  dimension. If $G=\SL(n)$ or $U$ is isotropic, i.e., $\omega|_U=0$ in
  case $G=\Sp(2n)$ and $q|_U=0$ in case $G=\SO(n)$, then the stabilizer of
  $U$ is a parabolic subgroup $P$ of $G$. The $G$-complete reducibility
  of $H$ implies that $H$ is contained in a Levi complement $L$ of
  $P$. 
Since $H$ is spherical, so is $L$, by Lemma \ref{lem:subgroups}. 
The maximality of $H$ implies $H=L$. 
Since all Levi subgroups lift to
  characteristic zero, it is easy to derive a list of spherical Levi
  subgroups from Krämer's list (cf.~\cite[Thm.\ 4.1]{Br}). One checks
  that all of them are contained in the table except for
  $\GL(n)\subset\SO(2n+1)$ and $\Gm\times\Sp(2n-2)\subset\Sp(2n)$
  which are not maximal.

  Now assume that $U$ is anisotropic. Then, in particular,
  $G\ne\SL(V)$.

  If $G=\Sp(2n)$, then $U\cap U^\perp\subsetneq U$ and therefore,
  $U\cap U^\perp=0$ by minimality of $U$. This means that $U$ is
  completely anisotropic and that $H$ is contained in a conjugate of
  $\Sp(2m)\times\Sp(2n-2m)$. The same reasoning works for $p\ne2$ and
  $G=\SO(n)$.

  So let $G=\SO(2n)$ and $p=2$. Then the associated bilinear form
  \[
  \omega_q(u,v)=q(u+v)-q(u)-q(v)
  \]
  is actually a symplectic form on $V$.  Again, if $U\cap
  U^{\perp_{\omega_q}}=0$, then $U$ is necessarily even dimensional
  and $H\subseteq\SO(2m)\times\SO(2n-2m)$.  But there is another
  possibility: $U$ is isotropic with respect to $\omega_q$ but
  $q|_U\ne0$.  Then $\omega|_U=0$ implies $q|_U=\ell^2$ where $\ell$ is
  an $H$-invariant linear form on $U$. By minimality of $U$ we have
  $\ker\ell=0$ and therefore $\dim U=1$.  Thus $H$ is a subgroup of
  $\SO(U^{\perp_{\omega_q}})\cong\SO(2n-1)$.
\end{proof}

\begin{corollary}
\label{cor:maximal}

Let $G$ be strictly classical and $H\subset G$ be a subgroup
  which is maximal among connected spherical $G$-completely reducible
  proper subgroups. Then either

\begin{itemize}
\item[(i)] $H$ is a maximal irreducible subgroup in
  Table~(\ref{eq:table1}), or

\item[(ii)] $H$ is contained in Table~(\ref{eq:table2}).

\end{itemize}

\end{corollary}

In the final lemmas of this section, we classify all spherical
$G$-completely reducible subgroups of the classical groups.

\begin{lemma}
\label{lem:lem22}
Let $G = \SL(n)$ with $n \ge 2$ and let $H \subset G$ be spherical,
$G$-completely reducible and reducible.  Then $H$ is
listed in Table \ref{table:sp-classical}.
\end{lemma}

\begin{proof}

  The assumptions on $H$ and Lemma \ref{lem:maxreducible} imply that
  $H \subseteq \Gm \cdot \SL(m) \cdot \SL(n-m)$, for $1 \le m \le
  n-1$.  We consider first the case $H = \Gm \cdot H_1 \cdot \SL(n-m)$
  where $H_1\subset\SL(m)$. Then, by induction on $\dim G$, we may
  assume that $H_1$ is contained in
  Table~\ref{table:sp-classical}. Then one checks that
  $H_1\subset\SL(m)$ lifts to characteristic zero. Hence $H$ is
  spherical if and only if it is in Krämer's list. This happens in a
  single case, namely $H=\Gm\times\Sp(n-1)\subset\SL(n)$ with $n\ge3$
  odd.

  By symmetry, we do not need to consider subgroups of the form
  $\Gm\cdot\SL(m) \cdot H_2$. Also no subgroups of the form
  $H=\Gm\cdot H_1\cdot H_2$ with $H_1\subset\SL(m)$ and
  $H_2\subset\SL(n-m)$ are spherical. Thus it remains to check
  $H=\Gm\cdot\SL(m)$ where the $\SL(m)$-factor is diagonally embedded
  into $\SL(m) \cdot \SL(m)$ with $n = 2m \ge 4$.  However, then $H$
  is not spherical, by \eqref{eq:2}.

  Finally, assume that $H = H'$ is semisimple.  Then $\Gm \cdot H'$ is
  one of the instances above.  As $H'$ lifts, it is contained in
  Krämer's table and thus covered by Corollary \ref{cor:lift}. The
  only cases of that form are $H=\SL(m)\cdot\SL(n-m)\subset\SL(n)$
  with $m\ne n-m$ and $H=\Sp(n-1)\subset\SL(n)$ with $n$ odd.
\end{proof}

\begin{lemma}\label{lem:lem23}

  Let $G = \Sp(2n)$ with $n \ge 2$ and let $H \subset G$ be spherical,
  $G$-completely reducible, and reducible.  Then $H$ is listed in
  Table \ref{table:sp-classical}.

\end{lemma}

\begin{proof}
  The assumptions on $H$ and Lemma \ref{lem:maxreducible} imply that
 either  $H \subseteq \GL(n)$ or $H \subseteq \Sp(2m) \times \Sp(2n-2m)$ for
  $0 < m < n$.

  In the first instance the inequality \eqref{eq:2} shows $H =
  \GL(n)$. In the second case, we first consider subgroups of the form
  $H=H_0\times\Sp(2n-2m)$ with $H_0\subset\Sp(2m)$. Then, by induction
  on $\dim G$, we may assume that $H_0$ is contained in
  Table~\ref{table:sp-classical}. Moreover, if $H_0\subset\Sp(2m)$
  lifts to characteristic zero, then $H$ is spherical if and only if
  it is in Krämer's list. One checks that there is a single case of
  that form, namely $H=\Gm\times\Sp(2n-2)$.

  Next, we consider those $H_0\subset\Sp(2m)$ which do not
  lift. This means $p=2$ and we have to deal with the following cases:

  (i) $H=\SO(2l)\cdot\Sp(2m-2l)\cdot\Sp(2n-2m)$ with $1\le l\le m<
  n$. Then $H$ is contained in the liftable subgroup
  $\Sp(2l)\cdot\Sp(2m-2l)\cdot\Sp(2n-2m)$ which is spherical if and
  only if one of the factors is trivial. Thus, for $H$ to be
  spherical, we need $l=m$. In that case, $H$ is indeed spherical,
  since then $H\subset G$ is isogenous to the liftable subgroup
  $\SO(2m)\cdot\SO(2n-2m+1)\subset\SO(2n+1)$.

  (ii) $H=\Spin(7)\cdot\Sp(2n-8)$ with $n\ge 5$ is never spherical, by
  Corollary \ref{cor:spin7}.

  (iii) $H=\G_2\cdot\Sp(2n-6)$ with $n\ge 4$ is only spherical for
  $n=4$, by Corollary \ref{prop:g2}.

  (iv) $H=\G_2\cdot\Sp(2)\cdot\Sp(2n-8)$ with $n\ge5$ is never
  spherical, since it is contained in the non-spherical subgroup
  $\Sp(6)\cdot\Sp(2)\cdot\Sp(2n-8)$.

  Now we discuss groups of the form $H=H_1\cdot
  H_2\subset\Sp(2m)\cdot\Sp(2n-2m)$. It follows from the discussion
  above that $H_i$ is one of $\Gm\subset\Sp(2)$ for $p\ne2$,
  $\SO(2m)\subset\Sp(2m)$ for $p=2$, or $\G_2\subset\Sp(6)$ for
  $p=2$. This leads to the following possibilities for $H$:

(i) $p\ne2$ and $H=\Gm\cdot\Gm\subset\Sp(4)$ which is not spherical,
by \eqref{eq:2}.

(ii) $p=2$ and $H=\SO(2m)\cdot\SO(2n-2m)$ with $1\le m<n$. In this
    case $H\subset G$ is isogenous to the liftable and non-spherical
    subgroup $\SO(2m)\cdot\SO(2n-2m)\subset \SO(2n+1)$.

(iii) $p=2$ and $H=\G_2\cdot\SO(2)\subset\Sp(8)$ which is not
    spherical, by \eqref{eq:2}.

  Finally, a general subgroup $H$ is obtained from a group of the form
  $H_1\cdot H_2$ by replacing one or several isogenous factors by
  diagonal subgroups. This is only possible in the following cases:

(i) $H\subset\Sp(2m)\cdot\Sp(2m)\subseteq\Sp(4m)$ with $m\ge1$. Then
$H$ is not spherical, by \eqref{eq:2}.

(ii) $H\subset\SO(4)\cdot\Sp(2)\subset\Sp(6)$. Again, $H$ is not
spherical, by \eqref{eq:2}.

  This finishes the proof of the lemma.
\end{proof}

\begin{lemma}\label{lem:lem24}

  Let $G = \SO(n)$, $n\ge2$ (with $n$ even if $p=2$) and let $H
  \subset G$ be spherical, $G$-completely reducible, and reducible.
  Then $H$ is listed in Table \ref{table:sp-classical}.

\end{lemma}

\begin{proof}

  Thanks to Lemma \ref{lem:maxreducible}, either $H \subseteq \GL(m) \subset
  \SO(n)$, for $n=2m \ge 2$, or $H \subseteq \SO(m) \cdot \SO(n-m)
  \subset \SO(n)$ where $1\le m <n$. For $p=2$, we may assume in the
  latter case that either both $m$ and $n$ are even or that $n$ is
  even and $m=1$.

  Assume first $H\subseteq\GL(m)\subset\SO(2m)$. Then the dimension
  estimate \eqref{eq:2} implies that the codimension of $H$ in
  $\GL(n)$ is at most $n$. Thus, the codimension of
  $H_0=(H\cap\SL(n))^\circ$ in $\SL(n)$ is at most $n$, as well. The list
  of maximal spherical subgroups of $\SL(n)$ (see
  Corollary~\ref{cor:maximal}) shows that $H_0=\Gm\subset\SL(2)$. Thus, the
  only instance is $H=\SO(2)\cdot\SO(2)\subset\SO(4)$.

  Now we treat the case $H \subseteq \SO(m) \cdot \SO(n-m)$ for
  $p\ne2$ or $p=2$ and $m,n$ even. First let
  $H=H_0\cdot\SO(n-m)\subset\SO(n)$, where $H_0\subset\SO(m)$. By
  induction, we may assume that $H_0$ is contained in Table
  \ref{table:sp-classical}. If $H_0$ is liftable, then sphericality can
  be checked with Krämer's table. It turns out that there is no
  instance of this type. On the other hand, there is only one
  non-liftable case, namely
  $H=\Delta_q\SL(2)\cdot\SO(n-4)\subset\SO(n)$ with $n\ge5$ and
  $q=p^s>1$. None of these subgroups is spherical: use inequality
  \eqref{eq:2} for $n=5,6,7$ and Lemma \ref{lemma:SOSP2} for $n\ge8$.

  The remaining case to consider is where $H$ is obtained from $\SO(m)
  \cdot \SO(n-m)$ by replacing some isogenous factors by a diagonal
  subgroup. Then either $H\subset\SO(m)\cdot\SO(m)\subset\SO(2m)$ with
  $m\ge2$ or $H\subset\SO(3)\cdot\SO(4)\subset\SO(7)$. None of these
  subgroups can be spherical, by \eqref{eq:2}.

  Now we treat the case when $p=2$, $n=2d$ is even and $m=1$, i.e.,
  $H\subset\SO(2d-1)\subset\SO(2d)$. There is a bijective isogeny
  $\SO(2d-1)\to\Sp(2d-2)$ and all $G$-completely reducible, spherical
  subgroups of $\Sp(2d-2)$ are known, by Lemma \ref{lem:lem23}.  Thus
  we arrive at the following cases:

(i) $H=\GL(2d-1)\subset\SO(2d)$ lifts and is not spherical.

(ii) $H=\SO(2)\cdot\SO(2d-3)\subset\SO(2d)$ lifts and is not spherical.

(iii) $H=\SO(2l-1)\cdot\SO(2d-2l+1)\subset\SO(2d)$ lifts and is spherical for
  all $1\le l\le d$.

(iv) $H=\SO(2l)\cdot\SO(2d-2l-1)\subset\SO(2d)$ lifts and is not spherical.

(v) $H=\Spin(7)\subset\SO(10)$ lifts and is not spherical.

(vi) $H=\G_2\subset\SO(8)$ lifts and is spherical.

(vii) $H=\G_2\cdot\SO(3)\subset\SO(10)$ is not spherical, by \eqref{eq:2}.

This finishes the proof of the lemma.
\end{proof}

This concludes our classification of the spherical, 
$G$-completely reducible subgroups of strictly classical groups.

\section{$G$-Completely Reducible, Spherical 
Subgroups of Exceptional Groups}

Throughout this section let $G$ be a simple group of exceptional type.

\begin{lemma} 
\label{lemma8}
Let $G$ be a simple group of exceptional type and $H\subset G$ a
subgroup which is maximal among proper, connected, $G$-completely
reducible, spherical subgroups of $G$.  Then one of the following
holds:
\begin{align*}
& G=\E_6 \text{ and } H\in\{\A_1\A_5,\Gm\cdot\D_5,\F_4,\C_4\ (p\ne2)\};\\
& G=\E_7 \text{  and } H\in\{\A_1\D_6,\A_7,\Gm\cdot\E_6\};\\
& G=\E_8 \text{ and } H\in\{\A_1\E_7,\D_8\};\\
& G=\F_4 \text{ and } H\in\{\A_1\C_3\ (p\ne2),\B_4,\C_4\ (p=2)\};\\
& G=\G_2 \text{ and } H\in\{\A_1\widetilde\A_1,\A_2,\widetilde\A_2\ (p=3)\}.
\end{align*}
\end{lemma}

\begin{proof}
  Assume first that $\rk H=\rk G$. If $p\ne2$ for $G=\F_4$ and $p\ne3$
  for $G=\G_2$ then $H$ is given by an additively closed subroot
  system. In particular, $H$ lifts to characteristic zero and the
  spherical cases can be read off of Krämer's list. Observe, that
  $H=\A_1\C_3$ in $G=\F_4$ is no longer maximal for $p=2$, since it is
  contained in a subgroup of type $\C_4$.

  Now suppose that $G=\F_4$ and $p=2$ or $G=\G_2$ and $p=3$, $\rk
  H=\rk G$, and that $H$ does not lift. Then the remaining
  possibilities for $H$ have been determined by Liebeck and Seitz,
  \cite[Table 10.4]{LieSei}, namely $(G,H)=(\F_4,\C_4)$ or
  $(G,H)=(\G_2,\widetilde \A_2)$. Using the inseparable isogeny of $G$
  in both cases, $H$ is mapped to a subgroup which lifts and is
  spherical. So $H$ itself is spherical in both instances.

  Finally, assume that $\rk H<\rk G$. Then we claim that $H$ is a
  maximal connected subgroup of $G$. Indeed, if $H$ were contained in
  a proper parabolic subgroup $P$  of $G$, then 
  the $G$-complete reducibility of $H$
  implies that $H$ lies in a Levi subgroup $L$ of $P$. Since $L$ is
  $G$-completely reducible as well, we get $H=L$, by maximality of
  $H$, contrary to our assumption on the rank of $H$. But $H$ cannot
  be a proper subgroup of a connected proper subgroup $K$ of $G$
  either, since $K$ would be $G$-completely reducible, hence also
  reductive. In fact, by the argument above, $K$ would not be contained
  in any proper parabolic subgroup of $G$. This finishes the proof of the
  claim.

  Now we know that $H$ is one of the subgroups of
  \cite[Table~1]{LieSei}. None of them is spherical for dimension
  reasons except for $(G,H)=(\E_6,\F_4)$ and $(G,H)=(\E_6,\C_4)\
  (p\ne2)$. In both cases, $H$ lifts and is spherical,
  cf.~\cite{springer} and \cite{Br}. Observe that for
  $p=2$, the group $H=\C_4$ is not maximal in $G=E_6$, 
since then it is contained in a subgroup of type
  $\F_4$.
\end{proof}

\begin{lemma}
\label{lem:e6f4}
  Let $G$ be a simple group of exceptional type and $H\subset G$ a
  proper, connected, non-maximal, $G$-completely reducible, spherical
  subgroup of $G$. Then one of the following holds:
  \begin{align*}
    & G=\E_6 \text{ and } H\in\{\D_5,\C_4\ (p=2)\}, \text{ or } \\
    & G=\F_4 \text{ and } H\in\{\A_1\B_3\ (p=2),\A_1\C_3\ (p=2)\}.
  \end{align*}
\end{lemma}

\begin{proof}
  Since $H$ is spherical in $G$, it satisfies the inequalities
  \[
  \rk H\le\rk G,\quad \dim H\ge{1\over2}(\dim G-\rk G).
  \]
  First we claim that, except for $G=\F_4$ and $p=2$ or $G=\G_2$ and
  $p=3$, we may assume that $\rk H<\rk G$. Indeed, otherwise $H$ lifts
  and would therefore be in Krämer's list. But one easily checks that
  there all maximal rank spherical subgroups are in fact maximal.

  Another constraint for $H$ is that it is a proper subgroup of one of
  the groups in Lemma \ref{lemma8}. It is now easy but 
  somewhat tedious to list
  all possible types for $H$ which match the requirements. We wind
  up with very few cases:

  (i) $G=\E_6$ and $H\in\{\D_5,\Gm\cdot\B_4,\ \B_4,\C_4\ (p=2)\}$: We
  claim that all subgroups of these types lift to characteristic 0:
  The subgroup $H=\D_5$ has to be the second factor in
  $\Gm\cdot\D_5$. Thus it lifts and is spherical. One checks that a
  group of type $\D_5$ contains a unique conjugacy class of subgroups
  of type $\B_4$, namely $\SO(9)\subset\SO(10)$. Thus $\Gm\cdot\B_4$
  inside $\Gm\cdot\D_5$ lifts and is not spherical. Also $H=\C_4$
  lifts, cf.\ \cite{Br} and is spherical. There are two possibilities
  for $H=\B_4$; either $H$ is inside $\Gm\cdot\D_5$ or else inside
  $\F_4$. In both cases, $H$ lifts and is not spherical.

  (ii) $G=\E_7$ and $H\in\{\D_6,\E_6\}$: Here $H=\D_6$ is normal in
  $\A_1\D_6$. Hence is lifts and is not spherical. Likewise, $H=\E_6$
  is normal in $\Gm\cdot\E_6$. Hence is lifts and is not spherical, as
  well.

  (iii) $G=\E_8$ and $H = \E_7$: Here $H=\E_7$ is normal in $\A_1\E_7$,
  hence it lifts and is not spherical.

  (iv) $G=\F_4$ and $H\in\{\A_1\C_3, \widetilde \A_1\B_3, \D_4\}$ and
  $p=2$: Let $H=\A_1\C_3$ or $H = \widetilde\A_1\B_3$. Without loss of
  generality we may assume that the positive root $\alpha$ in the
  $\A_1$-factor is a dominant weight of $\F_4$. Thus it is either the
  highest long or the highest short root. The roots orthogonal to
  $\alpha$ form a root system of type $\C_3$ or $\B_3$,
  respectively. Thus, $H=\A_1\C_3$ lifts to characteristic zero, while
  $H = \widetilde\A_1\B_3$ differs from the former by an inseparable
  isogeny of $\F_4$. So both are unique and spherical. There are two
  subgroups of type $\D_4$ corresponding to the two root subsystems
  consisting of all the long roots and all the
  short roots, respectively.  Stemming from a closed root subsystem,
  the first subgroup lifts, and thus so does the second, as it is
  obtained from the first by the isogeny of $G$. Thus, none of them is
  spherical.
\end{proof}

\section{Non-$G$-Completely Reducible, Reductive Spherical Subgroups}

Now we complete the classification by considering the
non-$G$-completely reducible subgroups of $G$. 
Throughout this section let 
$G$ be a connected reductive group over
$k$ and let $H\subseteq G$ be a non-$G$-completely reducible subgroup of $G$.
Then there exists a parabolic subgroup $P$ of $G$ containing $H$ so that
$H$ is in no Levi subgroup of $P$.
Indeed, there is a canonical such parabolic subgroup $P$ with this property only depending on $H$,
the so  called \emph{optimal destabilizing parabolic subgroup} 
associated with $H$, 
obtained by means of geometric invariant theory,  
cf.\ \cite[\S 5.2]{BMRT}.

It is convenient to use the characterization of parabolic subgroups of $G$
in terms of $1$-parameter subgroups of $G$, 
e.g.\ see \cite[2.1--2.3]{richardson} and 
\cite[Prop.\  8.4.5]{spr2}:

\begin{lemma} 
\label{lem:cochars} 
Given a parabolic subgroup $P$ of $G$ and any Levi subgroup $L$ of $P$,
there exists a $1$-parameter subgroup $\lambda$ of $G$ 
such that the following hold:
\begin{itemize}
\item[(i)]   $P = P_\lambda := \{g\in G \mid \underset{t\to 0}{\lim}\,
        \lambda(t) g \lambda(t)^{-1} \textrm{ exists}\}$.
\item[(ii)]  $L = L_\lambda := C_G(\lambda(\Gm))$.
\item[(iii)] The map $\pi = \pi_\lambda : P_\lambda \to L_\lambda$ defined by
\[
\pi_\lambda(g) := \underset{t\to 0}{\lim}\, \lambda(t)g \lambda(t)\inverse
\]
        is a surjective homomorphism of algebraic groups.
        Moreover, $L_\lambda$ is the set of fixed points of $\pi_\lambda$
        and $R_u(P_\lambda)$ is the kernel of $\pi_\lambda$.
\end{itemize}
\end{lemma}

\begin{remark}
\label{non-Gcr}
We note that $H \subset G$ is $G$-completely reducible if and only if 
for every $1$-parameter subgroup $\lambda$ of $G$ with 
$H \subset P_\lambda$, $H$ is $G$-conjugate to $\pi_\lambda(H)$,
see \cite[Lem.\ 2.17, Thm.\ 3.1]{BMR}, or \cite[Thm.\ 5.8(ii)]{BMRT}.
\end{remark}

Our first result shows that we can reduce the question of 
non-$G$-completely reducible, spherical subgroups of $G$ to 
ones that are $G$-completely reducible and spherical.
For this we use  again the Deformation
Theorem \ref{Theorem1}, this time with $S=\BBA^1_k=\Spec k[t]$.

\begin{proposition}
  \label{Proposition1}
  Let $G$ be a connected reductive group over $k$ and let $H\subseteq
  G$ be a reductive subgroup of $G$ lying in the parabolic subgroup $P
  = P_\lambda$ for some $1$-parameter subgroup $\lambda$ of $G$. Then
  $H$ is spherical in $G$ if and only if $\pi_\lambda(H)$ is.
\end{proposition}

\begin{proof}
  Define $\CH$ to be the closure of $\{(t,g)\mid
  t\in\Gm,\lambda(t)^{-1}g\lambda(t)\in H\}$ in $\BBA^1_k\times
  G$. This is a flat subgroup scheme of the trivial group scheme
  $\CG=\BBA^1_k\times G\to \BBA^1_k$ whose fiber $\CH_t$ is conjugate
  to $H$, for $t\ne0$,
cf.\ \cite [1.2.6, 1.2.7, 2.1.6]{bruhattits}.
Since $\pi_\lambda(h)=\lim_{t\to
    0}\lambda(t)h\lambda(t)^{-1}$ for all $h\in H$, we see that
  $\pi_\lambda(H)\subseteq\CH_0$. Since $H$ is reductive,
  $\Ker\pi_\lambda|_H=R_u(P)\cap H$ is finite and
  therefore $\dim \pi_\lambda(H)=\dim H$. Thus, also $\dim \pi_\lambda(H)=\dim
  \CH_0$ which implies $\pi_\lambda(H)^\circ=\CH_0^\circ$. Thus, our assertion
  boils down to showing that $\CH_0$ is spherical if and only if
  $\CH_1=H$ is spherical which follows immediately from Theorem
  \ref{Theorem1} with $S=\BBA^1_k$.
\end{proof}

We analyze the situation of Proposition \ref{Proposition1} further.

\begin{proposition} 
\label{Proposition2} 
Let $H\subseteq P=P_\lambda \subseteq G$ be as in Proposition
\ref{Proposition1} and assume that $H^* := \pi_\lambda(H) \subseteq
L=L_\lambda$ is not conjugate to $H$ inside $P$. Let $Z:=Z(L)^\circ$ be
the connected center of $L$. Then $Z\not\subseteq H^*$. In particular,
if $H$ is spherical, then $ZH^*$ is a reductive, non-semisimple,
spherical subgroup of $G$.
\end{proposition}

\begin{proof}
  Suppose $Z\subseteq H^*$.  Then, by Lemma \ref{lem:cochars}(ii),
  $C^*:=\lambda(\Gm)\subseteq H^*$. Let $C\subseteq H$ be the preimage
  of $C^*$ in $H$. Since $H\to H^*$ is an isogeny, $C$ is a
  1-dimensional torus lying in the center of $H$. Moreover, $C$ is a
  maximal torus of $C^*R_u(P)$ hence conjugate to $C^*$. So we may
  assume $C=C^*$. But then $H\subseteq C_G(C^*)=L$, and thus $H=H^*$,
  contradicting our assumptions.
\end{proof}

In the following lemma, we denote by $P_m$ the standard maximal
parabolic subgroup of the simple group $G$ corresponding to the $m$-th
simple root in the labeling of the Dynkin diagram of $G$ according to
\cite{bourbaki:groupes}.  Let $U = R_u(P)$ be the unipotent radical of
$P_m$.

\begin{lemma} 
\label{lemma9}
Let $G$ be a simple group and $H$ is a
connected, reductive, 
non-$G$-completely reducible, spherical subgroup of $G$
which is contained in the parabolic subgroup $P$ of $G$.
Then there are the following possibilities for $H^* = \pi_\lambda(H)$, $P$, 
and $G$ as in Proposition \ref{Proposition2}:
\begin{table}[h!t]
\renewcommand{\arraystretch}{1.5}
$$
  \begin{array}{llllc}
H^* & P & G & U &H^1_\gen(H',U)\\ 
\hline\hline
\SL(m){\times}\SL(n) &P_m,P_n & \SL(m{+}n)\ m>n\ge1 & k^m\otimes k^n &\begin{cases}k& m=2 \\ 0& m>2 \end{cases}\\ 
\Sp(2n) & P_1,P_{2n} & \SL(2n+1)\ n\ge2 & k^{2n} & k\\ 
\SL(2n+1) & P_{2n},P_{2n+1} & \SO(4n+2)\ n\ge2 & \wedge^2k^{2n+1} & 0\\ 
\D_5 & P_1,P_6 & \E_6 & k^{16} \text{ (half-spin reps.) } & 0
\end{array}
$$
\end{table}

In each case, the unipotent radical $U$ of $P$ is a vector group
on which $H^*$ acts linearly and irreducibly according to this
table. The last column lists the first generic cohomology group in the
sense of \cite{CPSK}.
\end{lemma}

\begin{proof}
The subgroups $H^*$ are just those $G$-completely reducible, spherical
subgroups which are centralized by a non-trivial torus, because 
this is a necessary condition, by Proposition \ref{Proposition2}. 
The cohomology groups have been calculated in, e.g., \cite{CPS}.
\end{proof}

We keep the notation of Lemma \ref{lemma9}. We know from the proof of
Proposition \ref{Proposition1} that the projection $\pi: H\to  H^*$ is an
isogeny. Its kernel $U\cap H$ is therefore a finite normal, hence
central subgroup of $H$. Moreover, $U\cap H$ is a $p$-group, 
since it is a subgroup of $U$. 
We conclude that $U\cap H = 1$, 
i.e., that $H\to H^*$ is bijective.

Now let $Q:=H\cdot U=H^*\ltimes U$. Our goal is to determine all
conjugacy classes of subgroups $H\subseteq Q$ such that the induced
projection $\pi:H\to H^*$ is bijective. If this bijection is even
an isomorphism, then it is well known this task is accomplished by the
cohomology group $H^1(H^*,U)$.

In general, we use the fact that each $H^*$ of interest is defined
over $\BBF_p$. This means that $H^*$ admits a Frobenius endomorphism
$F:H^*\to H^*$. Because $\pi$ is purely inseparable, it factors through 
a sufficiently high power $F^s$ of $F$, i.e., there is an $s\ge0$ and an
isogeny $\psi:H^*\to  H$ such that $F^s = \pi \circ \psi$.

Now let $\Qtilde$ be the fiber product of $Q$ over
$F^s:H^*\to H^*$. Then we have a cartesian diagram
\begin{equation*}
  \xymatrix{
    \Qtilde \ar[r]^{\tilde\pi} \ar[d] & H^* \ar[d]^{F^s} \\ 
    Q \ar[r]^{\pi} & H^*}  
\end{equation*}
Moreover, $\psi$ defines a section $\tilde\psi$ of $\tilde\pi$ such
that $H$ is the image of $\tilde\psi(H^*)$ in $Q$. Now observe that,
$\Qtilde=H^*\ltimes U^{(p^s)}$, where $U^{(p^s)}$ is the $s$-th
Frobenius twist of $U$. Therefore, the conjugacy class of $\tilde\psi$
and therefore of $H$ is determined by an element of
$H^1(H^*,U^{(p^s)})$.  By definition, the generic cohomology group
$H^1_\gen(H^*,U)$ is the inductive limit of the system
\begin{equation*}
  \xymatrix{
H^1(H^*,U) \ar[r] & H^1(H^*,U^{(p)}) \ar[r]  
& H^1(H^*,U^{(p^2)}) \ar[r]  & H^1(H^*,U^{(p^3)}) \ar[r] & \ldots}  
\end{equation*}
cf.\ \cite{CPSK}.
It is well known that elements of $H^1(H^*,U^{(p^s)})$ classify conjugacy
classes of (scheme theoretic) complements of $U^{(p^s)}$ in $\Qtilde$.
Thus, the conjugacy classes of the subgroups $H$ are classified by
elements of $H^1_\gen(H^*,U)$.

\begin{corollary} 
\label{cor:nonGcr}
Let $G$ be a simple group and $H\subseteq G$ a connected, reductive, 
spherical subgroup which is not $G$-completely reducible in
$G$. Then $G=\SL(2n+1)$ for some $n\ge1$ and $H=\SO(2n+1)$ or its dual
$H=\SO(2n+1)^\vee$.
\end{corollary}

\begin{proof}
  By the definition of $G$-complete reducibility there is a parabolic
  subgroup $P\subseteq G$ containing $H$ such that $H$ is not
  conjugate to $H^*=\pi_\lambda(H)$. From the discussion above we
  infer that $H^*\subset P$ is one of the cases in Lemma \ref{lemma9}
  with $H^1_\gen(H^*,U)\ne0$. Thus $G=\SL(2n+1)$ and $H^*=\Sp(2n)$
  with $n\ge1$. Because the centralizer $C_G(H^*) \cong \Gm$ of
  $\Sp(2n)$ acts non-trivially on $H^1_\gen(H^*,U)\cong k$, there
  exists only one conjugacy class of $H$ in $G$, depending on the
  choice of $P$, though. Thereby we obtain the two cases above.
\end{proof}

This concludes the proof of our main Classification Theorem: Using
Remark \ref{rem:isogeny}, 
we may assume that $G$ is either strictly classical or
exceptional. Then the $G$-completely reducible, connected, spherical
subgroups are determined in Corollary \ref{cor:maximal} and 
Lemmas \ref{lemma8}, \ref{lem:e6f4},
respectively. Finally, Corollary \ref{cor:nonGcr} 
lists all non-$G$-completely reducible subgroups.


\bigskip {\bf Acknowledgments}: 
The authors acknowledge 
support from the DFG-priority program 
SPP1833 ``Representation Theory''.
We are grateful to the anonymous referee for 
suggesting several improvements.


\bigskip

\bibliographystyle{amsalpha}

\end{document}